\documentclass[reqno,12pt]{amsart}
\usepackage{amsfonts}
\usepackage{amsthm}
\usepackage{amsmath}
\usepackage{amscd}
\usepackage[latin2]{inputenc}
\usepackage{t1enc}
\usepackage[mathscr]{eucal}
\usepackage{indentfirst}
\usepackage{graphicx}
\usepackage{graphics}
\usepackage{pict2e}
\usepackage{epic}
\usepackage{enumerate}
\numberwithin{equation}{section}
\usepackage[margin=2.9cm]{geometry}
\usepackage{epstopdf} 
\usepackage[hyphens]{url} 
\usepackage{fancyhdr}

 \usepackage{fancyhdr,lipsum}
 
 \usepackage{dirtytalk}

 \providecommand{\ts}{\hspace{15pt}}

\newcommand{\mb}{\mathbb}
\newcommand{\mc}{\mathcal}
\newcommand{\lap}{\triangle}
\theoremstyle{plain}
\newtheorem{Th}{Theorem}[section]
\newtheorem{Lemma}[Th]{Lemma}

 \theoremstyle{definition}

\newtheorem{?}[Th]{Problem}

\providecommand{\hs}{\hspace{15pt}}

\begin{document}
\title[Maximum Norm Estimates] {$ L^{\infty}$- estimates of the solution of the Navier-Stokes equation for non-decaying initial data}

\author{Santosh Pathak}

\address{University of New Mexico \\ 
Department of Mathematics and Statistics\\ 
Albuquerque, NM 87131,USA  }
\email{spathak@unm.edu}

\begin{abstract} In this paper, we derive the main result of a paper by H-O Kreiss and Jens Lorenz from a different approach than the method proposed in their paper. More precisely, we consider the Cauchy problem for the incompressible Navier-Stokes equations in $\mb R^n$ for $n\geq 3 $  with non-decaying  initial data and derive a priori estimates of the maximum norm of all derivatives of the solution in terms of the maximum norm of the initial data. This paper is also an extension of their paper to higher dimension. 
\end{abstract}
\maketitle
\section{Introduction}
\bigskip

 We consider the Cauchy problem of the Navier-Stokes equations in $\mb R^n$ for $n\geq 3$: 
\begin{align}
 u_t + u \cdot \nabla u + \nabla p = \triangle u, \ts \nabla \cdot u = 0, 
\end{align}
with initial condition 
\begin{align}
u(x,0)= f(x), \ts x \in  \mb R^n, 
\end{align}
where $u=u(x,t)= (u_1(x,t), \cdots u_n(x,t)) $ and $ p=p(x,t) $ stand for the unknown velocity vector field of the fluid and its pressure, while $f=f(x)=(f_1(x),\cdots  f_n(x) ) $ is the given initial velocity vector field. In what follows, we will use the same notations for the space of vector-valued and scalar functions for convenience in writing.
\medskip 

There is a large literature on the existence and uniqueness of solution of the Navier-Stokes equations in $\mb R^n $.  For given initial data solutions of (1.1) and (1.2)  have been constructed in various function spaces. For example, if $f\in L^r $ for some $r$ with $ 3 \leq r < \infty $, then it is well known that there is a unique classical solution in some maximum interval of time $0\leq t < T_f $ where $ 0 < T_f \leq \infty $. But for the uniqueness of the pressure one requires $| p(x,t)| \to 0$ as $|x| \to \infty$. See \cite{Kato} and \cite{Wiegner}  for $r=3 $ and \cite{Amann}  for $ 3 < r < \infty $.
\medskip 

If $f\in L ^{\infty}$ then existence of a regular solution follows from \cite{Cannon}. The solution is only unique if one puts some growth restrictions on the pressure as $| x| \to \infty$. A simple example of non-uniqueness is demonstrated in \cite{Kim} where the velocity $u$ is bounded  but $|p(x,t)| \leq C |x| $. In addition, an estimate $|p(x,t)|\leq C(1+|x|^{\sigma} ) $ with $ \sigma <1$ ( see \cite{Galdi} ) imply uniqueness. Also, the assumption $ p \in L^1_{loc}(0,T; BMO) $ (see \cite{G}) implies uniqueness. 
\medskip

In this paper we are interested in reproving the results of a paper by H-O Kreiss and J. Lorenz (see \cite{Lorenz}) for the initial data $ f \in L^{\infty}(\mb R^n) $ for $ n \geq 3$ using different approach than theirs in terms of dealing with the pressure term in the Navier-Stokes equations. The approach in this paper, to prove the principal result of the Kreiss and Lorenz paper, is more \say{functional analytic} approach in which the role of \say{the Leray projector} is being implemented to get rid of the pressure term from the Navier-Stokes equations. As a consequence of that, the details and techniques in obtaining some significantly complicated  results related to the pressure part in the Kreiss and Lorenz paper are being avoided which makes this paper different and simpler in that sense. At the same time, this paper is also an extension of the work by Kreiss and Lorenz to the higher space dimension whereas such  generalization, in the Kreiss and Lorenz paper by their approach, seems complicated because of the non-local nature of the pressure term in the Navier-Stokes equations. Since the main source of this paper is the Kreiss and Lorenz paper, it is appropriate to give some insight of their work in this paper as well.  Before we start outlining  some key aspects of their paper, we introduce the following notations and will be using them throughout this paper. 
\begin{align*}
| f|_{\infty} = \mathop{\sup}_{x} { |f(x)|}  \ts \text{with} \ts |f(x)|^2= \sum_{i} f_i^2(x), 
\end{align*}
and $D^{\alpha} =D_1^{\alpha_1} \cdots D_n^{\alpha_n},  D_i = \partial /{\partial x_i } \hspace{5pt} \text{for a multiindex} \hspace{5pt} \alpha =(\alpha_1,\cdots, \alpha_n )$.
In what follows, if  $ | \alpha | = j $, for any $j=0,1,\cdots $, then we will denote $D^{\alpha}=D_1^{\alpha} \cdots D_n^{\alpha_n} $ by $D^j$. We also set 
\begin{align*}
|\mathcal {D}^j u(t) |_{\infty} :=|\mathcal {D}^j u(\cdot, t ) |_{\infty} =\mathop{\max}_{|\alpha|=j}  |D^{\alpha} u(\cdot,t)|_{\infty} .
\end{align*}
Clearly, $|\mathcal{D}^j u(t) |_{\infty} $ measures all space derivatives of order $j$ in maximum norm.
\bigskip

Following theorem is the main result of the paper by Kreiss and Lorenz \cite{Lorenz} for $n=3$ which is also the principal result of this paper for $n\geq 3$.
\bigskip 
\begin{Th}
Consider the Cauchy problem for the Navier-Stokes equations (1.1), (1.2), where $ f \in L^{\infty} ( \mb R^n)$ and $  \nabla \cdot f = 0 $ is understood in the sense of distribution. There is a constant $c_0>0$ and for every $j=0,1,\cdots $ there is a constant $K_j$ so that 
\begin{align}
t^{j/2} | \mc D^j u(t) |_{\infty} \leq K_j | f|_{\infty} \ts \text{for} \ts 0 < t \leq  \frac{c_0}{|f|_{\infty}^2}.
\end{align}
The constants $c_0$ and $K_j$ are independent of $t$ and $f$.
\end{Th} 
\bigskip

Let us briefly discuss some key ideas of the Kreiss and Lorenz paper. Rewrite (1.1) as 
\begin{align*}
u_t = \triangle u + Q
\end{align*}
where 
\begin{align*}
Q= - \nabla p - u \cdot \nabla u. 
\end{align*}
Applying $D^j$ for $ j \geq 0 $ and using \textit{Duhamel's principle}, one obtains 

\begin{align}
v(t) = D^j e^{\triangle t } f + \int_0^t e^{\triangle (t-s) } D^j Q(s) ds,  \ts  v:= D^j u. 
\end{align}
Roughly speaking, obtaining the desired result of the Kreiss and Lorenz paper is a twofold in view of equation (1.4): first, estimates on the solution of the heat equation. Second, estimates on the derivatives of $Q$. Also, notice in (1.4), one can move one derivative $D$ to the heat semi-group and consequently requiring an estimate for $ | D^{j-1} Q |_{\infty} $ to estimate $ | v ( t)|_{\infty} $. Clearly, it is necessary to determine the pressure term $p$ of the Navier-Stokes equations so that $(u,p)$ solves (1.1) and (1.2);  the  estimate  of the derivatives of $p$ is being used to estimate $ | D^{j-1} Q |_{\infty}$. To proceed towards obtaining the required estimates on the pressure, Kreiss and Lorenz determine the pressure from the Poisson equation
\begin{align}
\triangle p = - \nabla \cdot ( u \cdot \nabla ) u 
\end{align}
which is given by 
\begin{align}
p = \sum_{i,j} R_i R_j ( u_i u_j) ,
\end{align}
where $R_i = (-\triangle )^{-1/2} D_i $ is the ith Riesz transform. Since the Riesz transforms are not bounded in $L^{\infty}(\mb R^n)$, the pressure term $p \in L^1_{loc} (0,T;BMO) $ where $ BMO$ is the space of functions of bounded mean oscillation. Because of the non-local nature of the pressure, the proof of Theorem 1.1 of the Kreiss and Lorenz paper  is complicated, however. This is where the method proposed in this paper deviates significantly from the approach adopted  by Kreiss and Lorenz in their paper \cite{Lorenz}.
\bigskip 

For the purpose of proving Theorem 1.1 for $n \geq 3 $, we start by transforming moemntum equation of  the Navier-Stokes equations  into the abstract ordinary differential equation for $u$
\begin{align}
u_t = \triangle u - \mb P (u \cdot \nabla ) u 
\end{align}
by eliminating the pressure, where $\mb P $ is the Leray projector  defined by 
\begin{align*}
\mb P = (  \mb P_{ij})_{1\leq i,j \leq n } , \ts  \mb P_{ij} = \delta_{ij} + R_i R_j; 
\end{align*}
where $ R_i$ is as in (1.6) and $ \delta_{ij}$ is the Kronecker delta function. Note that the equation (1.7) is obtained from (1.1 ) by applying the Leray projector with the properties $\mb P( \nabla p)=0, \mb P ( \triangle u ) = \triangle u$, since $\nabla \cdot u = 0 $. We use the solution operator $e^{ \triangle t } $ of the heat equation to transform the abstract differential equation into an integral equation 
\begin{align}
u(t) = e^{\triangle t} f -\int_0^t e^{ \triangle (t-s) } \mb P(u \cdot \nabla  u) (s) ds \ts t>0.
\end{align}
\medskip

In a paper by Giga and others \cite{Giga} for  $n \geq 2 $, they consider the initial data $f\in BUC (\mb R^n)$ which is the space of all bounded uniformly continuous functions or in $L^{\infty} (\mb R^n) $  which is the space of all essentially bounded  functions, and construct a unique local in time  solution of (1.8). Such solution of (1.8)  is called mild solution of (1.1) and (1.2). They later proved in the same paper that such mild solution is indeed a strong solution of the Navier-Stokes equations (1.1) and (1.2) in some maximum interval of time.  In addition, for essentially bounded initial data, existence and uniqueness of a solution of (1.1) and (1.2) is also proved in \cite{Cannon}; however, Giga and others in \cite{Giga} claim  that their approach is simpler than the method proposed in \cite{Cannon}.  In the same paper by Giga and others  \cite{Giga}, while constructing such mild solution of (1.1) and (1.2), it requires to obtain the estimate  $  t^{1/2}  | \nabla u |_{\infty} \leq C | f|_{\infty}  $ for some constant $C>0$ independent of $t$ and $f$ in some maximum interval of time.  However, such maximum norm estimates for higher order derivatives of the velocity field had not been achieved until H-O Kreiss and J. Lorenz  obtained in \cite{Lorenz}  for $ f \in L^{\infty} ( \mb R^3)$. 
\bigskip

The main work of this paper will focus on deriving estimate (1.3) of Theorem 1.1 by a \say{different approach} in a few ways  than that of the Kreiss and Lorenz paper adopts.  At the same time, this paper will also  demonstrate the fact, the absence of the pressure term in the transformed abstract differential equation (1.7) eliminates significant amount of work of the paper by Kreiss and Lorenz while obtaining the uniform estimates of the  pressure and its derivatives. However, there are some intriguing developments in the work of this paper due to the application of the Leray projector in our \say{different approach}. 
\bigskip

Major difficulty in proving Theorem 1.1 lies in the fact that the Leray projector $\mb P$ is not a bounded operator in $ L^{\infty}(\mb R^n)$, since the Riesz transforms are not bounded in this space although they are bounded in $L^r(\mb R^n)$ for $ 1 < r < \infty$. To overcome the difficulty, we obtain an uniform bound on the composite operator $ D^j e^{\lap t} \mb P$ for $ j= 0,1 \cdots $ in section 2.
\bigskip

This paper is organized in the following ways: In section 2 we introduce a few  estimates for the solution of the heat equations and state and prove a few lemmas and a corollary which are used later.  In section 3, for illustrative purpose, we introduce an analogous system and prove Theorem 3.1 which establishes result of Theorem 1.1 for the analogous system.  In section 4 we prove Theorem 1.1 using the same techniques as in the proof of Theorem 3.1.  Finally, in section 5, we outline some remarks on the use of the estimate (1.3)  obtained in Theorem 1.1.
\medskip 

\section{ Some Auxiliary Results} 
\bigskip
Let us consider $f\in L^{\infty} ( \mb R^n) $. The solution of 
\begin{align*}
u_t = \triangle u , \ts u = f \ts \text{at} \ts t=0,
\end{align*}
is denoted by 
\begin{align*}
u(t):= u(\cdot, t ) = e^{\triangle t }f = \theta(x,t) * f  
\end{align*}
where $\theta(t)= \theta (x,t)=1/(4\pi t )^{n/2} e^{-|x|^2/{4t}}, t>0 $ is the $n$ dimensional heat kernel in $\mb R^n$ and $*$ is the convolution operator. 
It is well known that 
\begin{align}
|e^{ \triangle t } f |_{\infty} &\leq |f|_{\infty}, \\
| \mc D^j e^{ \triangle t } f |_{\infty} &\leq C_j t^{-j/2} | f |_{\infty}, \ts t > 0, \ts j=1,2, \cdots 
\end{align}
Here, and in the following $C, C_j, c,$ etc are positive constants that are independent of $t$ and the initial function $f$.
\bigskip

\begin{Lemma}
Let $ \theta(t) = \theta(x,t)$ be the n-dimensional heat kernel in $ \mb R^n$. Then, for every $ j = 1,2 \cdots $  and every $ t>0,  D^j \theta (t) $  belongs to the Hardy space $ \mc H^1( \mb R^n ) $ and  
\begin{align}
|| D^j \theta(t) ||_{\mc H^1( \mb R^n )} \leq C_j t^{-j/2}. 
\end{align}
for some constant $C_j$.
\end{Lemma}
\begin{proof}
First, let us recall the definition of the Hardy space.
\begin{align*}
\mc H^1( \mb R^n ) = \{ u \in L^1(\mb R^n ) \ts s.t  \ts  \mathop{\sup}_{s>0} | h_s * u | \in L^1( \mb R^n ) \} 
\end{align*}
for some Schwartz class function $h$ where $h_s(x)= s^{-n} h(\frac{x}{s}), s>0$  such that $ 0 \leq h \leq 1 $ and $ \int h = 1 $. We may endow $ \mc H^1( \mb R^n )$ with the norm 
\begin{align*}
||u ||_{\mc H^1( \mb R^n )} = || u ||_{L^1(\mb R^n)} + || \mathop{\sup}_{s>0} | h_s * u | ||_{  L^1( \mb R^n ) }. 
\end{align*}
For any $j=0,1, \cdots $, we want to prove $ D^j \theta(t)  \in {\mc H^1( \mb R^n )}$, that means, it suffices to show that $\mathop{\sup}_{s>0} | h_s * D^j \theta(t) | \in L^1( \mb R^n )$. For that, we take $ h(x) = \theta(x,1)$ and notice 
\begin{align*}
| h_s * D^j \theta (t) (x) | = | D^j \theta(t+ s  , x) | 
\end{align*}
so that 
\begin{align*}
\mathop{\sup}_{s>0} | h_s * D^j \theta(t) | = | D^j \theta(t) | \in L^1(\mb R^n ). 
\end{align*}
Finally, we arrive at  
\begin{align*}
|| D^j \theta(t) ||_{\mc H^1( \mb R^n )} \leq C_j t^{-j/2}, \ts t>0.
\end{align*}
\end{proof}

\begin{Lemma}
For any  $f \in L^{\infty}( \mb R^n)$. Let  $j \geq 1$, there is a constant $C_j$ such that 
\begin{align}
| \mc D^j e^{\triangle t}  \mb P f |_{\infty} \leq C_j t^{-j/2} | f |_{\infty}  \ts \text{for} \ts 0 < t \leq T. 
\end{align}
\end{Lemma}
\begin{proof}
For $ 1 \leq i \leq n$ and $ t > 0$, by the definition of the Leray projector, we write
\begin{align*}
  (D^j e^{ \triangle t}  \mb P f)_i  & = D^j e^ {\triangle t}  f_i + \sum_{l=1}^n D^j e^{\triangle t}  R_i R_l f_l \\
& = D^j \theta(t) * f_i + \sum_{l=1}^n D^j R_i R_l \theta(t) * f_l \\
& = \sum_{l=1}^n ( \delta_{il} + R_iR_l ) ((D^j \theta(t) ) * f_l)  \\
& = \sum_{l=1}^n k_{il}(t) * f_l
\end{align*}
where the kernel $k_{il} (t) = ( \delta_{il} + R_i R_l ) (D^j \theta(t))$. Since the Riesz transforms are bounded in $ \mc H^1( \mb R^n )$  and $ || \cdot ||_{L^1(\mb R^n ) } \leq || \cdot ||_{\mc H^1( \mb R^n )} $, we have 
\begin{align*}
|| k_{il} (t) ||_{L^1(\mb R^n) } & \leq || k_{il} (t) ||_{\mc H^1( \mb R^n )}  \\
& \leq || D^j \theta (t) ||_{\mc H^1( \mb R^n )}. 
\end{align*}
Thus, from previous Lemma 2.1 we obtain 
\begin{align*}
|| k_{il} (t) ||_{L^1(\mb R^n) } \leq C_j t^{-j/2} \ts \text{for} \ts t>0. 
\end{align*}
Finally, by the Young's inequality of convolution we estimate as 
\begin{align*}
|( D^j e^{\triangle t }  \mb P f)_i |_{\infty} &  \leq \sum_{l=1}^n | k_{il} * f_l |_{\infty}  \\
& \leq C || k_{il} (t) ||_{L^1(\mb R^n)} |f_l|_{\infty} \\
& \leq C_j t^{-j/2} |f_l|_{\infty}. 
\end{align*}
Hence, Lemma 2.2 is proved.
\end{proof}
\bigskip

 
  
 \section{Estimates For the System $u_t = \triangle u + D_i \mb P  g$}
\bigskip

In this section we state and prove an analogous theorem  of Theorem 1.1 for the solution of an illustrative system. For that purpose, let us recall $ \mb P ( u \cdot \nabla ) u = \sum_i D_i \mb P ( u_i u)$ for  $  1 \leq i \leq n $. Therefore, it is appropriate to consider the  illustrative system to be 
\begin{align}
u_t = \triangle u + D_i \mb P ( g(u(x,t))) , \ts x\in \mb R^n , \ts t\geq 0 
\end{align}
with initial function 
\begin{align}
u(x,0) =f(x) \ts \text{where} \ts f\in L^{\infty} (\mb R^n). 
\end{align}
Here  $ g: \mb R^n \to \mb R^n $ is assumed to be quadratic in $u$. We will prove the maximum norm estimates of the derivatives of the solution of (3.1) and (3.2) by the maximum norm estimate of the initial function $f$. It is well-known that  the solution is $C^{\infty}$ in a maximal interval $ 0 < t < T_f$ where $ 0 < T_f \leq \infty$.

\bigskip

\begin{Th} Under the assumptions on $f$ and $g$ mentioned above, the solution of  (3.1) and (3.2) satisfies the following:

\begin{enumerate}[(a)]
\item There is a constant $c_0 >0$ with 
\begin{align*}
T_f > \frac{c_0}{| f|_{\infty}^2}
\end{align*}
and 
\begin{align}
| u (t)|_{\infty} \leq 2 | f |_{\infty} \ts \text{for} \ts 0 \leq t \leq \frac{c_0}{| f |_{\infty}^2}.
\end{align}
\item For every $ j = 1, 2, \cdots , $ there is a constant $K_j >0$ with 
\begin{align}
t^{j/2} | \mc D^j u(t) |_{\infty} \leq K_j | f |_{\infty}  \ts \text{for} \ts 0 < t \leq \frac{c_0}{| f |_{\infty}^2}.
\end{align}
\end{enumerate}
The constants $c_0$ and $K_j$ are independent of $t$ and $f$.
\end{Th}
\bigskip 

Proof of part (a) will be given in the following lemma, and consecutively, we will also derive  the estimate (3.4) of part (b). Consider $u$ as the solution of the inhomogeneous heat equation $ u_t = \lap u + D_i \mb P F $ where  \begin{align*}
    F(x,t) := g(u(x,t)) \ts \text{for} \ts x \in \mb R^n,  \ts 0 \leq t < T_f. 
\end{align*}
Since $ g $ is quadratic in $u$, there is a constant $C_g$ such that we have the following:
\begin{align}
| g(u) | \leq C_g |u|^2 , \ts |g_u(u)| \leq C_g |u| \ts \text{for all} \ts u \in \mb R^n.
\end{align}
Next lemma estimates the maximum norm of $u$.
\bigskip

\begin{Lemma}
Let $C_g$ denote the constant in (3.5) and let $C$ denote the constant in (2.6); set $c_0 =\frac{1}{16C^2 C_g^2} $. Then we have $T_f > c_0/{ | f |_{\infty}^2 } $ and 
 \begin{align}
 | u (t) |_{\infty} < 2 | f |_{\infty} \ts \text{ for } \ts 0 \leq t < \frac{c_0}{|f|_{\infty}^2 }.
 \end{align}
 \end{Lemma}
 \begin{proof}
 Suppose (3.6) does not hold, then we can find the smallest time $t_0$ such that $ |u(t_0) |_{\infty} = 2 | f |_{\infty} $. Since $t_0$ is the smallest time so we have $ t_0 < c_0/|f|_{\infty}^2 $. Now by (2.1) and (2.6) we have 
 \begin{align*}
 2|f|_{\infty}  & = |u(t_0) |_{\infty} \\
 & \leq |f|_{\infty} + C t_0^{1/2} \mathop{\max}_{0\leq s \leq t_0 } | g(s)  |_{\infty} \\
 & \leq |f|_{\infty} + C C_g t_0^{1/2} \mathop{\max}_{0 \leq s \leq t_0}  | u (s) |_{\infty}^2  \\
 & \leq | f|_{\infty} + C C_g t_0^{1/2} 4 |f |_{\infty}^2 .
 \end{align*}
 This gives 
 \begin{align*}
 1 \leq 4 C C_g t_0^{1/2} | f |_{\infty}, 
 \end{align*}
 therefore $ t_0 \geq 1/{( 16 C^2 C_g^2 | f |_{\infty}^2)}= c_0/{|f|_{\infty}^2} $ which is a contradiction. Therefore (3.6) must hold. The estimate $ T_f > c_0/{| f |_{\infty}^2} $ is valid since $ \limsup_{t \to T_f} | u(t) |_{\infty} = \infty $ if $ T_f $ is finite.
 \end{proof}
 \bigskip 

Now we prove estimate (3.4) of Theorem 3.1 by induction on $j$. Let $j \geq 1 $ and assume 
 \begin{align}
 t^{k/2} | \mc D^k u (t) |_{\infty} \leq K_k | f|_{\infty}, \ts \text{for} \ts  0 \leq  t \leq  \frac{c_0}{|f|_{\infty}^2}  \ts \text{and} \ts 0 \leq   k \leq j-1. \end{align}
 where $c_0$ is the same constant as in the previous lemma. 
 Next, we begin by applying $ D^j$ to the equation $u_t = \triangle u + D_i \mb P g(u) $ to obtain 
 \begin{align*}
 v_t = \triangle v + D^{j+1} \mb P g(u), \ts  v:= D^j u ,\\
 v(t) = D^j e^{\triangle t } f + \int_0^t e^{\triangle (t-s)} D^{j+1} ( \mb P g(u))(s) ds .
 \end{align*}
 Using (2.2) we get 
 \begin{align}
 t^{j/2} |v(t)|_{\infty} \leq C |f|_{\infty} + t^{j/2} \biggl  |\int_0^t e^{\triangle (t-s)} D^{j+1} ( \mb P g(u)) (s) ds \biggr |_{\infty}.  
 \end{align}
 We split the integral into 
 \begin{align*}
 \int_0^{t/2} + \int_{t/2}^t =: I_1 + I_2 
 \end{align*}
 and obtain 
\begin{align*}
 |I_1(t)|  & = \biggl | \int_0^{t/2} D^{j+1} e^{\triangle (t-s) } ( \mb P g(u))(s) ds \biggr |_{\infty} \\
 & \leq \int_0^{t/2} |D^{j+1} e^{\triangle (t-s)} (\mb P g (u )) (s) ds |_{\infty} ds .
 \end{align*}
 Using the inequality in Lemma 2.2, we get 
 \begin{align*}
 | I_1(t) |_{\infty} & \leq C  \int_0^{t/2} (t-s)^{-(j+1)/2} |g(u(s))|_{\infty} ds \\
 & \leq  C |f |_{\infty}^2 t^{(1-j)/2}. \\
 \end{align*}
 The integrand in $I_2$ has singularity at $s=t$. Therefore, we can move only one derivative from $D^{j+1} \mb P  g(u)$ to the heat semigroup.( If we move two or more derivatives then the singularity becomes non-integrable.) Thus, we have 
 \begin{align*}
 |I_2(t)|_{\infty}   = \biggl | - \int_{t/2}^t De^{\triangle (t-s)} (D^j  \mb P g(u))(s)  ds \biggr | _{\infty}. \nonumber 
 \end{align*}
 Since the Leray projector commutes with any derivatives, therefore 
 \begin{align*}
 |I_2(t) |_{\infty}  = \bigg| - \int_{t/2}^t D e^{\triangle (t-s) } ( \mb P D^j g (u))(s) ds \bigg|_{\infty}.
 \end{align*}
 If we use Lemma 2.2 for $j=1$, we obtain 
 \begin{align}
 |I_2(t) |_{\infty} \leq C \int_{t/2}^t (t-s)^{-1/2} |D^j g(u(s)) |_{\infty} ds. 
 \end{align}
 Since $g(u)$ is quadratic in $u$, therefore
 \begin{align*}
 |D^j  g(u) |_{\infty} \leq C | u|_{\infty} | \mc D^j u |_{\infty} +  \sum_{k=1}^{j-1} | \mc D^k u |_{\infty} | \mc D^{j-k} u |_{\infty} .
 \end{align*}
 By induction hypothesis (3.7) we obtain 
 \begin{align}
 \sum_{k=1}^{j-1} | \mc D^k u (s)  |_{\infty} | \mc D^{j-k} u (s)  |_{\infty} \leq C s^{-j/2} | f |_{\infty}^2. 
 \end{align}
 Integral (3.9) can be estimated as below: 
 \begin{align*}
  |I_2(t) |_{\infty}  &\leq C \int_{t/2}^t (t-s)^{-1/2} \bigg(  C | u(s)|_{\infty} | \mc D^j u(s)  |_{\infty} +  \sum_{k=1}^{j-1} | \mc D^k u (s)  |_{\infty} | \mc D^{j-k} u (s)  |_{\infty}  \bigg) ds \\
&= J_1 + J_2. 
 \end{align*}
Using (3.10), and since  $ \int_{t/2}^t (t-s)^{-1/2} s^{-j/2} ds = C t^{(1-j)/2}$, where $C$ is independent of $t$, we obtain $|J_2(t) |_{\infty} \leq C | f |_{\infty}^2 t^{(1-j)/2}$.
 \medskip \\ 
 For $J_1$, we have 
 \begin{align*}
 | J_1(t) |_{\infty}  &= C \int_{t/2}^t (t-s)^{-1/2} |u(s)|_{\infty} | \mc D^j u(s) |_{\infty} ds  \\
 & \leq C |f|_{\infty} \int_{t/2}^t (t-s)^{-1/2} s^{-j/2} s^{j/2} |\mc D^j u (s) |_{\infty} ds \\
 & \leq C |f|_{\infty} t^{(1-j)/2} \mathop{\max}_{0 \leq s \leq t } \{ s^{j/2} \mc D^j u(s) |_{\infty} \}. 
 \end{align*}
 We use these bounds to bound the integral in (3.8). We have $v= D^j u $. Then maximizing the resulting estimate for $t^{j/2} |D^j u(t)|_{\infty}$ over all derivatives $D^j$ of order $j$ and setting 
\begin{align*}
\phi (t) := t^{j/2} |\mc D^j u(t)|_{\infty} 
\end{align*}
and from (3.8), we obtain the following estimate 
\begin{align*}
\phi(t) \leq C |f|_{\infty} +C t^{1/2} |f|_{\infty}^2 +C |f|_{\infty} t^{1/2} \mathop{\max}_{0\leq s \leq t} \phi (s) \ts \text{for} \ts 0 \leq t \leq \frac{c_0}{|f|_{\infty}^2 }.
\end{align*}
Since $t^{1/2} |f|_{\infty} \leq \sqrt{c_0} $ then $C t^{1/2}  |f|_{\infty}^2 \leq C \sqrt{c_0} |f|_{\infty} $. Therefore
\medskip
\begin{align}
\phi (t) \leq C_j |f|_{\infty} + C_j |f|_{\infty} t^{1/2} \mathop{\max}_{0\leq s \leq t} \phi (s)  \ts \text{for} \ts 0 \leq   t \leq c_0/{|f|_{\infty}^2 }.
\end{align}
Let us fix $C_j$ so that the above estimate holds and set 
\begin{align*}
c_j = \min \bigg\{ c_0 , \frac{1}{4 C_j^2} \bigg\}.
\end{align*}
\medskip
First, let us prove the following
\begin{align*}
\phi(t) < 2 C_j |f|_{\infty}  \ts \text{for} \ts 0 \leq  t <  \frac{c_j }{|f|_{\infty}^2}.
\end{align*}
Suppose there is a smallest time  $t_0$ such that $ 0 < t_0 < c_j/{|f|_{\infty}^2 }  $ with $\phi(t_0) = 2C_j |f|_{\infty} $. Then using (3.11) we obtain
\begin{align*}
2 C_j |f|_{\infty} = \phi (t_0) \leq C_j |f|_{\infty} + 2 C_j^2  |f|_{\infty}^2 t_0^{1/2} ,
\end{align*}
thus 
\begin{align*}
1\leq 2 C_j |f|_{\infty} t_0^{1/2}  \ts \text{gives} \ts t_0 \geq c_j/{|f|_{\infty}^2}  
\end{align*}
which contradicts the assertion. Therefore, we proved the estimate 
\begin{align}
t^{j/2} |\mc D^j u(t) | _{\infty} \leq 2 C_j |f|_{\infty}   \ts \text{for} \ts 0 \leq  t \leq  c_j/{|f|_{\infty}^2 }.
\end{align}
If 
\begin{align}
T_j:= \frac{c_j}{|f|_{\infty}^2} < t \leq \frac{c_0}{|f|_{\infty}^2}=: T_0
\end{align}
then we start the corresponding estimate at $t-T_j$. Using Lemma 3.2, we have $|u(t-T_j)|_{\infty} \leq 2 |f|_{\infty}$ and obtain 
\begin{align}
T_j^{j/2} |\mc D^j u(t) |_{\infty} \leq 4 C_j |f|_{\infty} . 
\end{align}
Finally, for any $t$ satisfying (3.13) 
\begin{align*}
t^{j/2} \leq T_0^{j/2} = \bigg( \frac{c_0}{c_j} \bigg)^{j/2} T_j^{j/2}
\end{align*}
and (3.14) yield 
\begin{align*}
t^{j/2} |\mc D^j u(t) |_{\infty} \leq 4 C_j  \bigg( \frac{c_0}{c_j} \bigg)^{j/2} |f|_{\infty}.
\end{align*}
This completes the proof of Theorem 1.1. 

\section{Estimates For the Navier-Stokes Equations}
\bigskip

Recall the transformed abstract ordinary differential equation (1.7)
\begin{align}
u_t = \triangle u - \mb P (u \cdot \nabla  u), \ts \nabla \cdot u = 0 
\end{align}
with 
\begin{align}
   u(x,0)=f(x).
\end{align}
Solution of (4.1) and (4.2) is given by 
\begin{align}
u(t) = e^{\triangle t } f -\int_0^t  e^{ \triangle (t-s) } \mb P(u \cdot \nabla u)(s)  ds .
\end{align}
\medskip \\
 Using the solution (4.3) with previous  estimates  (2.1),(2.2) and (2.4) we prove the following lemma.
\bigskip

\begin{Lemma}
Set 
 \begin{align}
 V(t)= |u(t)|_{\infty} + t^{1/2} |\mc D u(t) |_{\infty} , \hs 0 <  t < T(f). 
 \end{align}
 \medskip 
 There is a constant $C>0$, independent of $t$ and $f$, so that 
 \begin{align}
 V(t) \leq C|f|_{\infty} + C t^{1/2} \mathop{\max}_{0 \leq s \leq t }{V^2(t)}  , \hs 0 < t < T(f). 
 \end{align}
 \end{Lemma}
  \begin{proof} Using estimate (2.1)  of the heat equation  in (4.3), we obtain
 \begin{align*}
 |u(t)|_{\infty}  & \leq |f|_{\infty} + \bigg| \int_0^t e^{\triangle(t-s) }  \mb P (u \cdot \nabla u )(s) ds \bigg|_{\infty}. 
 \end{align*}
 Apply identity $\mb P (u\cdot \nabla u ) = \sum_i D_i \mb P (u_i u )  $ with  the fact,  heat semi-group commutes with $D_i$, then use of the inequality (2.4) in  Lemma 2.2 for $j=1$ to proceed
\begin{align*}
 | u ( t)|_{\infty} & \leq  |f|_{\infty}  + C \int_0^t (t-s)^{-1/2} |u(s)|_{\infty}^2 ds \\
 &  =  |f|_{\infty} + C \int_0^t (t-s)^{-1/2}  s^{-1/2} s^{1/2} |u(s)|_{\infty}^2 ds \\
& \leq |f|_{\infty} + C \mathop{\max}_{0\leq s \leq t } \{ s^{1/2} |u(s)|_{\infty}^2  \} \int_0^t (t-s)^{-1/2}  s^{-1/2} ds.
\end{align*}
Since $  \int_0^t (t-s)^{-1/2}  s^{-1/2} ds = C>0 $ which is independent of $t$,  we have the following estimate 
 \begin{align}
 |u(t)|_{\infty}  & \leq | f|_{\infty} +  C  \mathop{\max}_{0\leq s \leq t } \{ s^{1/2} |u(s)|_{\infty}^2 \}  \nonumber \\
 | u(t)|_{\infty}  & \leq |f|_{\infty} + C t^{1/2}  \mathop{\max}_{0\leq s \leq t }{V^2(s) }. 
 \end{align}
 Apply $D_i$ to (4.1) to obtain
 \begin{align}
 v(t) = D_i e^{\triangle t} f - \int_0^t e^{ \triangle (t-s)  } D_i \mb P (u \cdot \nabla ) u(s) ds .
 \end{align}
We can estimate the integral in (4.7) using Lemma 2.2 for $j=1$ in the following way: 
\begin{align*}
  \bigg|\int_0^t D_i e^{ \triangle (t-s)}  \mb P (u \cdot \nabla  u)(s) ds \bigg| &  \leq  \int_0^t |D_i e^{ \triangle  (t-s)}  \mb P(u\cdot \nabla u) (s) | ds \\
 & \leq C \int_0^t (t-s)^{-1/2} |u(s)|_{\infty} | \mc D u(s)|_{\infty} ds \\
 &= C \int_0^t (t-s)^{-1/2} s^{-1/2} s^{1/2} |u(s)|_{\infty} |\mc D u(s) |_{\infty} ds \\
 & \leq C \mathop{\max}_{0 \leq s \leq t } { \{s^{1/2} | u(s) |_{\infty} |\mc D u(s)|_{\infty}\}  }   \int_0^t (t-s)^{-1/2}  s^{-1/2} ds \\
 & \leq C \mathop{\max}_{0\leq s \leq t } { \{ |u(s)|_{\infty}^2 + s | \mc D u(s) |_{\infty}^2 \} } .
\end{align*}
Therefore, we arrive at 
\begin{align}
|v(t) |_{\infty}  & \leq C t^{-1/2} |f|_{\infty} + C  \mathop{\max}_{0\leq s \leq t } { \{ |u(s)|_{\infty}^2 + s | \mc D u(s) |_{\infty}^2 \} }  \nonumber  \\
t^{1/2} |\mc D u(t) |_{\infty}  & \leq C |f|_{\infty} + C t^{1/2}  \mathop{\max}_{0\leq s \leq t } {V^2(t) }. 
\end{align}
 Using  (4.6) and (4.8), we have proved Lemma 4.1. 
\end{proof}
\bigskip

\begin{Lemma}
Let $C>0$  denote the constant in estimate (4.5) and set 
\begin{align*}
c_0 = \frac{1}{16C^4}.
\end{align*}
Then $T_f  > c_0 /{|f|_{\infty}^2}$  and 
\begin{align}
| u(t) |_{\infty} + t^{1/2} | \mc D u (t) |_{\infty} < 2C |f|_{\infty} \ts  \text{for}  \ts 0 \leq t < \frac{c_0}{|f|_{\infty}^2}. 
\end{align}
\end{Lemma}
\begin{proof}
We prove this lemma by contradiction after recalling the definition of $V(t)$ in (4.4). Suppose that (4.9) does not hold, then denote by $t_0$ the smallest time with $V(t_0)= 2C |f|_{\infty}$. Use (4.5) to obtain 
\begin{align*}
2C| f |_{\infty}  & = V(t_0) \\
&  \leq C |f|_{\infty} + C t_0^{1/2} 4 C^2 | f |_{\infty}^2, 
\end{align*}
thus 
\begin{align*}
1 \leq 4 C^2 t_0^{1/2} | f |_{\infty}^2 ,
\end{align*}
therefore $t_0 \geq c_0/ | f |_{\infty}^2$. This contradiction proves (4.9) and $T_f > c_0/{|f |_{\infty}^2}$.
\end{proof}
\medskip 
Lemma 4.2 proves Theorem 1.1 for $j=0$ and $j=1$. By an induction argument as in the proof of Theorem 3.1 one proves Theorem 1.1 for any $j =0,1,\cdots $

\section{Remarks}
\bigskip

We can apply estimate (1.9) of Theorem 1.1  for 
\begin{align}
\frac{c_0}{2 | f |_{\infty}^2} \leq t \leq \frac{c_0}{ | f |_{\infty}^2}
\end{align}
and obtain 
\begin{align}
| \mc D^j u (t) |_{\infty} \leq C_j | f |_{\infty}^{j+1} 
\end{align}
 in interval (5.1). Starting the estimate at $ t_0 \in [0, T_f) $ we have 
 \begin{align}
 | \mc D^j u ( t_0 + t ) |_{\infty} \leq C_j | u(t_0) |_{\infty}^{j+1} 
 \end{align}
 for 
 \begin{align}
 \frac{c_0}{2 | u(t_0)|_{\infty}^2} \leq t \leq \frac{c_0}{ | u(t_0)|_{\infty}^2}.
 \end{align}
Then, if $t_1$ is fixed with 
\begin{align}
\frac{c_0}{2 | f|_{\infty}^2} \leq t_1 < T_f,
\end{align}
we can maximize both sides of (5.3) over $ 0 \leq t_0 \leq t_1$ and obtain
\begin{align}
\max \bigg\{ | \mc D^j u(t) |_{\infty } : \frac{c_0}{2 | f|_{\infty}^2} \leq t \leq t_1+ \tau \bigg\} \leq C_j \max \{ | u(t) |_{\infty}^{j+1} : 0 \leq t \leq t_1 \} 
\end{align}
with 
\begin{align*}
\tau = \frac{c_0}{| u (t_1) |_{\infty}^2} 
\end{align*}
\medskip \\
Estimate (5.6) says, essentially, that the maximum of the $j$-th derivatives of $u$ measured by $ | \mc D^j u|_{\infty} $ , can be bounded in terms of $ | u |_{\infty}^{j+1} $. Clearly, a time interval near $t=0$ has to be excluded on the left-hand side of (5.6) for smoothing to become effective. The positive value of $\tau$ on the left-hand side of (5.6) shows that $ |u|_{\infty}^{j+1}$ controls $ | \mc D^j u |_{\infty}$ for some time into the future. 
\medskip  \\
As is well  known, if $(u, p)$ solves the Navier-Stokes equations and $ \lambda >0$ is any scaling parameter, then the functions $ u_{\lambda}, p_{\lambda} $ defined by 
\begin{align*}
u_{\lambda} (x,t) = \lambda u( \lambda x, \lambda^2 t) , \ts  p_{\lambda} ( x,t) = \lambda^2 p( \lambda x , \lambda^2 t ) 
\end{align*}
also solve the Navier-Stokes equations. Clearly,
\begin{align*}
| u_{\lambda} (t) |_{\infty} = \lambda | u ( \lambda^2 t)|_{\infty} , \ts | \mc D^j u_{\lambda} (t) |_{\infty} = \lambda^{j+1} | \mc D^j u ( \lambda^2 t ) |_{\infty}. 
\end{align*}
Therefore, $ | \mc D^j u |_{\infty} $ and $ | u|_{\infty}^{j+1} $ both scale like $ \lambda^{j+1} $, which is, of course, consistent with the estimate (5.6). We do not know under what assumptions $ | u |_{\infty}^{j+1} $ can conversely be estimated in terms of $ | \mc D^j u |_{\infty} $.

\end{document}